\documentclass[11pt,a4paper]{amsart}
\usepackage{amsthm, a4wide}
\usepackage{amsfonts,bbm}
\usepackage{amsmath}
\usepackage{amssymb}
\usepackage[english]{babel}
\usepackage[active]{srcltx}

\newtheorem{theorem}{Theorem}[section]
\newtheorem{corollary}[theorem]{Corollary}
\newtheorem{lemma}[theorem]{Lemma}
\newtheorem{proposition}[theorem]{Proposition}

\theoremstyle{definition}

\newtheorem{remark}{Remark}

\newcommand{\eps}{\varepsilon}
\let\e=\varepsilon

\let\O=\Omega
\let\ol=\overline
\let\ul=\underline
\newcommand{\ox}{\overline{x}}

\newcommand{\oa}{\overline{\alpha}}

\def\limn{\lim_{n\to\infty}}

\def\N{\mathbb{N}}

\def\R{\mathbb{R}}

\newcommand{\Mi}{\mathbf{M}}
\newcommand{\tr}{\operatorname{\text{tr}}}
\newcommand{\om}{\Omega}
\newcommand{\de}{\partial}

\DeclareMathOperator{\dist}{dist}
\DeclareMathOperator\diam{diam}

\begin{document} 

 \title[On a parabolic HJB equation degenerating at the boundary]{On a  parabolic Hamilton-Jacobi-Bellman equation degenerating at the boundary}
 \author{Daniele Castorina}\thanks{Dipartimento di Matematica, Universit\`a
    di Padova, Via Trieste 63, 35121 Padova, Italy,\\ email:
castorin@math.unipd.it}\author{ Annalisa Cesaroni}\thanks{Dipartimento di Scienze Statistiche, Universit\`a
    di Padova, Via Cesare Battisti 141, 35121 Padova, Italy, email:
   annalisa.cesaroni@unipd.it} \author{ Luca Rossi }\thanks{Dipartimento di Matematica, Universit\`a
    di Padova, Via Trieste 63, 35121 Padova, Italy,\\ email:
   lucar@math.unipd.it}

 \thanks{D.C. was partially supported by project {\sl Bando Giovani Studiosi 2013 - Universit\` a di Padova - GRIC131695}. A.C. and L.R. were partially supported by the GNAMPA  Project 2015 Processi di diffusione degeneri o singolari legati al controllo di 
dinamiche stocastiche}
\date{}
 \begin{abstract}
We derive the long time asymptotic of solutions to an evolutive
Hamilton-Jacobi-Bellman equation in a bounded smooth domain, in connection with
ergodic problems recently studied in \cite{bcr}. Our main assumption is an
appropriate degeneracy condition on the  operator  at the boundary. This
condition is related to the characteristic boundary points for linear operators
as well as to the irrelevant points for the generalized Dirichlet problem, and
implies in particular that no boundary datum has to be imposed.  We prove that
there exists a constant $c$ such that the solutions of the evolutive problem
converge uniformly, in the reference frame moving with constant velocity $c$, to
a unique steady state solving a suitable ergodic problem.
\end{abstract} 
\maketitle

\section{Introduction} \label{intro}
We are concerned with the asymptotic behavior as $t\to +\infty$ 
of solutions of the evolutive Hamilton-Jacobi-Bellman equation 
 \begin{equation}
 \label{paraintro}   
 u_t
 + \sup_{\alpha\in A} \left(-b(x,\alpha)\cdot Du
 - \tr (a(x,\alpha) 
D^2u
)-l(x,\alpha)
\right)=0,\quad x\in\om,\ t>0, 
\end{equation}
with bounded initial data $u(x,0)=u_0(x)$. The domain $\Omega\subset\R^N$ is
assumed to be bounded and smooth; no boundary condition is imposed, but just the
following control on the growth:

 \begin{equation}
 \label{BC}
\exists \; \lambda>0,\quad
\lim_{x\to \partial\O} u(x,t) d(x)^\lambda =0\qquad \text{loc. unif. in }t \geq
0
 \end{equation}

where $d(x)$ is the distance of $x \in \Omega$ from $\partial\Omega$. 

The main assumption is that the fully nonlinear elliptic operator 
 \begin{equation}
 \label{F0} 
F[u]:= \sup_{\alpha\in A} \left(-b(x,\alpha)\cdot Du(x)- \tr 
\left(a(x,\alpha)D^2 u(x)\right)\right) ,
 \end{equation}
degenerates in the normal direction to the boundary $\partial\Omega$, for all $\oa\in A$,   and 
the quantity $b(x,\oa)\cdot Dd(x) +\tr(a(x,\oa) D^2d(x))$ is positive near 
$\partial\Omega$ (see Assumption \eqref{invariance}). \\

This condition is related to the  invariance of the set $\Omega$ for the 
controlled diffusion process associated with the operator $F$ (see \cite{bcr}).
It allows  us to prove existence and uniqueness of a smooth solution to the Cauchy
problem associated with \eqref{paraintro}, without
imposing any boundary condition on $\partial\O$, but just the control on the
growth (see Theorem \ref{escauchy}). 

Once the well posedness of the Cauchy problem is established, we investigate
the large time behavior.
Our main result states that there exists a unique constant $c$, called the
ergodic constant, which governs the large time behavior of solutions, in the
following sense:
for every bounded continuous initial datum $u_0$, there exists a constant $K$,
depending only on $u_0$, such that the unique smooth solution to the Cauchy
problem satisfies   
\[u(x,t)+ct-\chi(x)+K\to 0 \qquad \text{as $t\to +\infty$, \
uniformly in $x\in\om$}.\]
For the precise statement we refer to Corollary \ref{main} at the end of the
paper.
The constant $c$ and the function $\chi$ are uniquely defined as the solution 
of the so called ergodic problem (or additive eigenvalue problem), that is,

\begin{equation}
 \label{cell1}\begin{cases}  
\sup_{\alpha\in A} \left(-b(x,\alpha)\cdot D\chi(x)- \tr (a(x,\alpha) 
D^2\chi(x))-l(x,\alpha)
\right)=c,\qquad   x\in\Omega\\ \chi\in L^\infty(\Omega), \
\sup\chi=0.\end{cases}
\end{equation} 

Recently in \cite{bcr} it has been proved that there exists a unique $c$ 
such that the first equation in \eqref{cell1} admits a smooth solution (unique
up to additive constants) satisfying the growth condition \[\lim_{x\to
\partial\O}\frac{\chi(x)}{\log d(x)} =0.\] In this paper, we refine this result,
showing that actually the solution $\chi$ is bounded in $\Omega$ by using
appropriate bounded barriers at the boundary of $\Omega$ (see Proposition
\ref{chibounded}).
Moreover, in Proposition \ref{chibounded} we derive  some regularity estimates 
of the solution $\chi$ to \eqref{cell1} up to the boundary of $\Omega$, which in
particular imply H\"older regularity of $\chi$ up to the boundary in the 1D
case.
An interesting open problem is to determine under which conditions $\chi$
is Lipschitz-continuous up to the boundary. 
Such regularity cannot be expected in general under our assumptions, as shown
in Remark \ref{regolaritachi} in Section \ref{bdry}. Analogous
regularity results for
solutions of ergodic problems for linear operators with singular drift in
bounded domains have been obtained in \cite{lp}. 
Moreover we recall that the generalized Dirichlet problem and 
the state constraint problem for Hamilton-Jacobi-Bellman operators  in bounded domains  has been studied in \cite{bb,il,ll}. 

Our methods are mainly based on comparison principle, 
strong maximum principle and careful estimates of solutions to
\eqref{paraintro} up to the boundary of $\Omega$ (see Proposition
\ref{estimateboundary}). 

Finally we recall that large time asymptotic of solutions to fully nonlinear
parabolic equations have been studied in the periodic setting by Barles and Souganidis in \cite{bs}. 
More recently, the large time behavior in the periodic setting for possibly degenerate Hamilton-Jacobi equations has been treated in \cite{cgmt,ln}.
Results on the large time behaviour of solutions in bounded domains have been obtained by Da Lio with Neumann boundary conditions in \cite{dl}, 
and with state constraint boundary conditions by Barles, Porretta and Tchamba in \cite{bpt}. 
\\

The paper is organized as follows. In Section \ref{hyp} we specify our
assumptions and set up convenient notations for the development of our study.
Section \ref{lyapbarr} is devoted to the explicit construction of 
Lyapunov functions and bounded barriers. In Section \ref{cauchy}
we study the well posedness of the Cauchy problem associated with
\eqref{paraintro} as well as some \textit{ad hoc} comparison principles for
sub/super solutions satisfying mild growth conditions at the boundary. Next, in
Section \ref{bdry} we   apply these results to the study of the boundary
behavior and the sharp regularity of the solution to \eqref{paraintro}. Finally,
we establish in Section \ref{conv} our main result about the large time
convergence towards a steady state solving a suitable additive eigenvalue
problem.

\section{Assumptions and notations} \label{hyp}
Throughout the paper we will assume, if not otherwise stated, that $\O$ be a 
bounded domain in $\R^N$ with $C^2$ boundary. 
Let $d(x)$ be the signed distance function to $\partial\Omega$, i.e. 
\begin{equation*}
d(x):=\dist(x, \R^n\backslash\Omega)-\dist(x, \overline{\Omega}).
\end{equation*} 
We know, from e.g.~\cite[Lemma 14.16]{gt}, that $d$ is of class $C^2$ in some 
neighborhood $\overline{\om}_\delta$ of the boundary, where, here and in the sequel, 
\begin{equation}\label{odelta}
\Omega_\delta:=\{x\in\Omega \ |\ d(x)<\delta\}.
\end{equation} 
We introduce the fully nonlinear homogeneous operator  
\begin{equation}\label{F} 
F[u]:= \sup_{\alpha\in A} \left(-b(x,\alpha)\cdot Du(x)- \tr 
\left(a(x,\alpha)D^2 u(x)\right)\right).
 \end{equation} 
and the Hamilton-Jacobi-Bellman operator 
 \begin{equation}
 \label{h}   
H[u]:=\sup_{\alpha\in A} \left(-b(x,\alpha)\cdot Du(x)- \tr (a(x,\alpha) 
D^2u(x))-l(x,\alpha)\right).
\end{equation}
where $A $ is a complete metric space and
\[b,l:\overline{\Omega}\times A\to \R^N,\qquad a:\overline{\Omega}\times A\to 
\Mi_{N\times N} 
\]
are bounded and continuous, $\Mi_{N\times N}$ being the space of  
$N\times N$ real matrices. 
We further assume $a(x,\alpha)$ to be symmetric and nonnegative definite for  
all $x,\alpha$.  
This implies that  $a \equiv \sigma\sigma^T$
for some    $\sigma: \overline{\Omega}\times A\to \Mi_{N\times r}$,  $r\in\N$. 

The main regularity assumptions on the coefficients of the operator are the
following: 
there exist $B>0$, $\eta \in (0,1]$, $\beta\in (1/2,1]$ such that,  for all 
$x,y\in\overline\O$ and $\alpha\in A$,
\begin{equation}
\label{reg3}  |b(x,\alpha)-b(y,\alpha)|, |l(x,\alpha)-l(y,\alpha)|\leq 
B|x-y|^\eta
\end{equation}
\begin{equation}\label{reg2} 
|\sigma(x,\alpha)-\sigma(y,\alpha)|\leq 
B|x-y|^\beta,
\end{equation}
where, even for matrices, $|\cdot|$ stands for the standard Euclidean norm. 
The regularity assumption on $a$ 
is given in terms of
its square root $\sigma$ as it is natural for applications  to 
stochastic control problems. We recall that if  $a(\cdot, \alpha)\in W^{2, 
p}(\Omega)$ with $\|a(\cdot,\alpha)\|_{W^{ 2,p}}\leq C$ for some $p>2N$ and 
$C>0$ independent of $\alpha\in A$, then it has a square root $\sigma$
satisfying \eqref{reg2} 
with  $\beta= 1-\frac{N}{p}$.
   
We assume that the operator is elliptic in the interior of $\O$, in the strong sense
 \begin{equation}
\label{sell} \text{$a(x,\alpha)>0$\quad for 
all $x\in\Omega$ and $\alpha\in 
A$  }\end{equation}
and that it degenerates at the boundary according to the following condition:
\begin{equation}
\label{invariance}
\begin{array}{c}
\exists \;\delta,k>0,\ \gamma<2\beta-1\leq 1,\quad
\text{such that for all }\overline{x}\in\partial \Omega\text{ and }\alpha\in
A,\\
\begin{cases}
\sigma^T (\ox,\alpha) Dd(\ox)=0 \\
b(x,\alpha)\cdot Dd(x)+\tr \left(a(x,\alpha)D^2 d(x)\right)
\geq k\, d^\gamma(x), & x\in\Omega\cap B_\delta(\overline x). 
\end{cases}
\end{array}
\end{equation} 

The first condition in \eqref{invariance} means that at any boundary point, the 
normal is a direction of degeneracy for $F$.
The second condition can be rewritten as: $F[d]\leq - k\, d^\gamma$ in a 
neighborhood of $\partial\Omega$; it is guaranteed if at the boundary 
the normal component of 
the drift points inward and is sufficiently large. 
Notice however that  condition \eqref{invariance} does not prevent the function 
$b(\cdot,\alpha)\cdot Dd(\cdot)+\tr(a(\cdot,\alpha) D^2d(\cdot))$ from vanishing at 
the boundary.  

We recall that \eqref{invariance} is a sufficient condition for  the invariance of the domain $\O$ 
for the stochastic control system with drift $b$ and diffusion $\sigma$ (see Prop. 6.5 in \cite{bcr}).


\section{Lyapunov functions and barriers at the boundary} \label{lyapbarr}
In this section we show that  under condition \eqref{invariance} the function $V(x)=d(x)^{-\lambda}$, for $\lambda
>0$, plays the role of a Lyapunov function for the system (see also \cite{bcr}). 
\begin{lemma}
\label{liap}
Assume that \eqref{reg3}, \eqref{reg2}, \eqref{invariance} hold. Then 
 for every $M\geq 0$ and every $\lambda>0$, there exists    $\delta>0$ such
that 
$d$ is of class $C^2$ in the set $\Omega_\delta$ defined by \eqref{odelta} and
there holds
\[-F[d^{-\lambda}]\leq F[-d^{-\lambda}]\leq-M \qquad 
\text{in }\;\Omega_\delta.\]
\end{lemma} 
\begin{proof} The first inequality immediately follows from the definition of
$F$. 
For the second one we take $\delta$ small enough so that $d\in
C^2(\Omega_\delta)$ and we compute, for $x\in \Omega_\delta$,
\[F[-d^{-\lambda}]= 
\frac{\lambda}{d^{\lambda+1}}\sup_{\alpha\in A} \left(-b(x,\alpha)\cdot Dd-
\tr 
\left(a(x,\alpha)D^2 d \right)+ \frac{\lambda +1}{d}|\sigma(x,\alpha)Dd|^2
\right).\]
Using \eqref{reg2},  \eqref{invariance} and  choosing $\bar x\in\partial\O$ such that $Dd(x)=Dd(\bar x)$, we get for every $\alpha$ 
\begin{equation}\label{sigma}|\sigma(x,\alpha)Dd(x)|^2=|
(\sigma(x,\alpha)-\sigma(\bar x, \alpha))Dd(\bar 
x)|^2 \leq B^2 |x-\bar x|^{2\beta}=B^2 d^{2\beta}(x). \end{equation}
Then we obtain, using  \eqref{invariance} and recalling that $\gamma<2\beta-1$,
\begin{eqnarray*} F[-d^{-\lambda}]&\leq&
\frac{\lambda}{d^{\lambda+1}}\left(F[d]+B^2 (\lambda
+1)d^{2\beta-1}\right)
\\ &\leq& \frac{\lambda}{d^{\lambda+1}} \left(- k  d^{\gamma
}+B^2(\lambda+1)d^{2\beta-1} \right),\end{eqnarray*}
which is smaller than any given $-M$ in $\Omega_\delta$, provided $\delta$ is
sufficiently small.
\end{proof} 
In the following we will also need the existence of strict supersolutions to 
$F=0$ in a neighborhood of the boundary of $\O$ which are not explosive at the 
boundary. 
\begin{lemma}\label{barrier} Assume that \eqref{reg3}, \eqref{reg2}, 
\eqref{invariance} hold.
Let $\rho\in (0, 1-\gamma)$, where $\gamma$ is the constant appearing in condition \eqref{invariance}. Then 
for every $M>0$ there exists $\delta$ small enough such that the function
$1-d^\rho$ is $C^2 (\O_\delta)$ and satisfies 
\[-F[1-d^\rho ]\leq 
F[d^\rho-1]\leq -M \qquad \text{in }\;\O_\delta.\] 
\end{lemma} 
\begin{proof}
As before, we prove the second inequality, since the first comes from the
definition of~$F$. 
For $\delta \in (0,1)$ small enough we have that the function $ d^\rho-1$ is
of class 
$C^2$ in $\O_\delta$, where it satisfies
\[F[ d^\rho-1]=\rho d^{\rho-1}\sup_\alpha \left( - b(x,\alpha)\cdot 
Dd- \frac{\rho-1}{d}|\sigma(x,\alpha)Dd|^2-  
\tr(a(x,\alpha)D^2d)
\right).\]
 Then, recalling that 
$0<\rho<1$ and the computation \eqref{sigma},  we obtain \[F[d^\rho-1]\leq \rho
d^{\rho-1}\left(F[d]+  
(\rho-1)B^2 d^{2\beta-1}\right)
.\] Hence,  using \eqref{invariance} and choosing $\rho<1-\gamma$, we see that
\[F[ d^\rho-1]\leq \rho d^{\gamma+\rho-1}\left(-k +   
(\rho-1)B^2 d^{2\beta-\gamma-1}\right)\leq -M   \qquad \text{in }\;\O_\delta
, \]  
provided $\delta$ is sufficiently small (depending in particular  on $k, B, 
\beta, \gamma, \rho, M$).\end{proof} 


\section{The Cauchy problem} \label{cauchy}
We study the following Cauchy problem in the 
cylinder $\O\times (0, +\infty)$: 
 \begin{equation}\label{Cauchy} 
  \begin{cases}
   u_t+H[u]=0, & x\in\Omega,\ t>0\\
u(x,0)=u_0(x), & x\in\Omega.
  \end{cases}
 \end{equation}
The initial condition is understood to hold in the classical sense, with 
$u_0\in C(\O)\cap  L^\infty(\O)$. Notice that no boundary condition on 
$\partial\O$ is 
imposed.

The first result is a comparison principle between smooth sub and supersolution
which satisfy an appropriate growth condition at the boundary. 

\begin{lemma}\label{lem:comparison}
Let $\ul u,\ol u\in C^{2,1}(\O\times (0,T])\cap C(\O\times [0,T])$ be
respectively a sub 
and a supersolution to \eqref{Cauchy} such that
$$
\exists \; \lambda>0,\quad
\limsup_{x\to\partial\om}  \ul u(x,t) d(x)^{\lambda}\leq 0\leq
\liminf_{x\to\partial\om} \ol u(x,t)d(x)^{\lambda}
\quad\text{uniformly in $t\in[0,T]$}.$$
Then $\ul u\leq\ol u$ in $\O\times (0,T]$.
\end{lemma}

\begin{proof}
We start with observing that
\begin{equation}\label{HF}
\forall u,w\in C^2,\quad
H[u]-F[w]\leq H[u-w]\leq
H[u]+F[-w].
\end{equation}
The first inequality implies that the function $w:=\ol u-\ul u$ satisfies
$$w_t+F[w]\geq \ol u_t+H[\ol u]-\ul u_t-H[\ul u]\geq 0 ,\qquad x\in\Omega,\
t>0.$$
Furthermore, $w$ is nonnegative at $t=0$ and fulfills the same condition as 
$\ol 
u$ at $\de\O$.
We know from Lemma \ref{liap} that $F[-d(x)^{-\lambda}]<-1$ in some 
$\O_\delta$. Fix $T>0$ and call
$$m:=\min\left\{\min_{(\O\setminus \O_\delta)\times[0,T]}w
\,,\,0\right\}.$$
For $\e>0$, the function $V_\e(x):=m-\e d(x)^{-\lambda}$ 
satisfies $F[V_\e]<0$ in $\O_\delta$, as well as
$V_\e(x)<m\leq w(x,t)$ if $d(x)=\delta$, and also
$$\liminf_{x\to\partial\om}(w(x,t)-V_\e(x))=+\infty
\quad \text{ uniformly in }t\in[0,T].$$ 
The latter implies that, for $\delta'\in(0,\delta)$ small enough, 
$w(x,t)>V_\e(x)$ if $d(x)=\delta'$ and $t\in[0,T]$.  Finally, observe that 
$w\geq V_\e$
at initial time. 
We can therefore apply the standard parabolic comparison principle (see, e.g.,
\cite{Lieb})
in the cylinder  $(\Omega_\delta\setminus\Omega_{\delta'})\times[0,T]$ and infer that 
$w\geq V_\e$ there. Due to the arbitrariness of $\delta'$ and $\eps$, 
this implies that $w\geq m$ in $\O_\delta\times[0,T]$,
whence
$$\inf_{\O\times[0,T]}w\geq
\min\left\{\min_{(\O\setminus\O_\delta)\times[0,T]}w
\,,\,0\right\}.$$
If the above right-hand side were negative, since $w\geq0$ at $t=0$, it would
be 
reached 
at some $(\ul x,\ul t)\in (\O\setminus\O_\delta) \times(0,T]$, and therefore, 
by the parabolic strong 
maximum principle (see \cite{Lieb}), $w$ would coincide with a negative constant for 
$t\leq\ul t$, which is impossible. This shows that $w\geq0$ in $\O\times[0,T]$.
\end{proof}

\begin{theorem} \label{escauchy}
For any $u_0\in L^\infty(\Omega)\cap C(\Omega)$, problem \eqref{Cauchy} admits a
unique solution $u\in C^{2,1}(\Omega\times (0, +\infty))\cap 
C(\Omega\times [0,+\infty))$ satisfying \eqref{BC}. Moreover, $u\in  
L^\infty(\Omega\times (0,T))$ for every $T>0$.
\end{theorem} 


\begin{proof} 
We start by proving existence and interior regularity. Let $\O^n:=\O\setminus 
\overline{\O}_{1/n}$ (according to definition \eqref{odelta}), for $n$ 
sufficiently large, so that $\O_n$  is smooth. Let $\xi$ be a standard mollifier 
with support contained in the unit ball and $\xi_n(x):=n^N\xi(n x)$ for $n\in\N$. 
We define $u_{0,n}:= u_0\ast \xi_n$ and consider  the following Cauchy-Dirichlet 
problem 
\begin{equation}\label{CauchyDirichlet} 
\begin{cases}
v_t+H [v]=0, & x\in\Omega_n,\ t>0\\
v(x,t)=u_{0,n}(x), & x \in \partial \Omega_n, t > 0\\
v(x,0)=u_{0,n}(x), & x \in \Omega_n 
\end{cases}
\end{equation} 
By \cite[Theorem 14.18]{Lieb}, problem 
\eqref{CauchyDirichlet} admits a unique solution $u_n \in C^{2,1}(\Omega_n 
\times (0, +\infty))\cap C (\overline{\Omega}_n \times[0,+\infty))$. 

Let us now fix $T>0$ and a compact set $Q \subset \Omega \times (0,T)$.
Then $Q\subset\Omega_n\times(0,T)$ for $n$ larger than some $\overline{n}$.
Thanks to our assumptions and a covering argument, 
we can apply Theorem 1.1 in \cite{tw} (see also Remark 1.1 parts (a) and (b) for 
the regularity issues regarding $u$ and $H$ respectively) in order to see that 
there exist some constants $\theta\in (0,1]$ and $C>0$ not depending on $n$ 
such that
\begin{equation}\label{stimaxun}
\forall n>\overline{n},\quad
\| D^{2} u_n \|_{C^{\theta,\theta/2}(Q)}\leq C.
\end{equation}
%
Notice that in principle $C$ depends also on $n$ through the uniform bound $\| u_n
\|_{L^{\infty} (\Omega_n\times [0,T])}$, but we can actually substitute
this bound 
 with the uniform bound $\|u_0\|_{L^{\infty} (\Omega)}$, independent 
of $n$, because $\pm \| 
u_{0,n} \|_{L^{\infty} (\Omega_n)} \pm \|l\|_\infty t$ are sub/super solutions
of \eqref{CauchyDirichlet} and thus the standard comparison principle yields
\begin{equation}\label{boundedness}\| u_n \|_{L^{\infty} (\Omega_n\times 
[0,T])} 
\leq   \| u_{0,n} \|_{L^{\infty} (\Omega_n)}+\|l\|_\infty T \leq   \| u_0 
\|_{L^{\infty} (\Omega)}+\|l\|_\infty T.\end{equation}
Now, from \eqref{CauchyDirichlet}, \eqref{stimaxun} and the regularity of the
coefficients, it follows that the $(\partial_t u_n)_{n>\overline{n}} $ are 
uniformly
H\"older-continuous in $Q$ in both the $x$, $t$ variables.
%
%
The Ascoli-Arzel\`a theorem eventually implies that there is a subsequence of 
$(u_n)_{n\in\N}$ converging in $C^{2+\zeta,1+\zeta/2}(Q)$, $\zeta < 
\theta$, to a function $u$ satisfying the first equation of \eqref{Cauchy} in 
$Q$. Finally, by a diagonal argument, we find a subsequence of 
$(u_n)_{n }$ for which this convergence holds true in any compact subset of
$\Omega_n\times(0,T)$.
Notice that, the limit being just local, we lose the information about 
the boundary and the initial behavior of the solution. 
Observe nevertheless that by \eqref{boundedness}, we get that $u$ is bounded in $\Omega\times [0,T]$. 

We claim now that $u\in C(\Omega\times [0,+\infty))$ and satisfies 
$u(x,0)=u_0(x)$, that is, it solves~\eqref{Cauchy}. 
Fix $x_0\in \Omega$. For $0<\eps<\dist(x_0,\partial\Omega)$ consider 
the function $u_{0,n}^\eps\in C(\Omega)$ defined by 
\[u_{0,n}^\eps(x)=\max_{|z-x_0|\leq \eps}
u_{0,n}(z)+\frac{2}{\eps^2}|x-x_0|^2\|u_{0,n}\|_\infty.\] 
We compute
\[H[u_{0,n}^\eps]\geq  -\frac{4\|u_{0,n}\|_\infty}{\eps^2}\left(\|b\|_\infty 
\diam(\om)+N\|a\|_\infty\right)-\|l\|_\infty .\]
Then, we take $M\geq  \|b\|_\infty 
\diam(\om)+N\|a\|_\infty$ and define the function 
\[u_{n}^\eps(x,t)=u_{0,n}^\eps(x)+\frac{4M}{\eps^2}\|u_{0,n}\|_\infty t +\|l\|_\infty t.\] 
It is easy to check, using our choice of $M$ and noticing that 
$u_{n}^\eps(x,t)\geq u_{0,n}^\eps(x)\geq u_{0,n}(x)$ in $\overline{\O}_n \times
[0,T]$, that $u_{n}^\eps(x,t)$ is a supersolution to \eqref{CauchyDirichlet}.
Then, by  comparison, we get
\[u_n(x,t)\leq u_{n}^\eps(x,t) \qquad x\in \Omega_n,\ t\geq 0. \] 
Computing the previous inequality at $x=x_0$ and letting $n\to +\infty$, we obtain 
\[ u(x_0,t)\leq \max_{|z-x_0|\leq \eps}
u_{0}(z)+\frac{4M}{\eps^2}\|u_{0}
\|_\infty t+\|l\|_\infty t \qquad \ t\geq 0.\] Now, letting $t\to 0$, we get
\[\limsup_{t\to 0}u(x_0,t)\leq \max_{|z-x_0|\leq \eps}
u_{0}(z).\] Finally letting $\eps\to 0$, the continuity of $u_0$ yields
\[\limsup_{t\to 0}u(x_0,t)\leq
u_{0}(x_0).\] 

Arguing in an analogous way, with the function 
\[u_{n}^\eps(x,t)=\min_{|z-x_0|\leq \eps}
u_{0,n}(z)-\frac{2}{\eps^2}|x-x_0|^2\|u_{0,n}\|_\infty-\frac{4M}{\eps^2}\|u_{0,n}\|_\infty t-\|l\|_\infty t\]
we also obtain  that  \[\liminf_{t\to 0}u(x_0,t)\geq
u_{0}(x_0).\] 
We conclude by the arbitrariness of $x_0$. Finally, uniqueness follows from
Lemma \ref{lem:comparison}.
\end{proof}



\section{Boundary behavior of the solution.} \label{bdry}

We investigate now the behavior of the solution $u$  to 
\eqref{Cauchy} at the boundary of $\O$. 
\begin{proposition}\label{estimateboundary} Let $u$ be the solution to 
\eqref{Cauchy}, \eqref{BC} provided by Theorem \ref{escauchy}.  
Then for every $\rho\in (0, 1-\gamma)$, there exists $\bar\delta\in (0,1)$ such  
that for $\delta<\bar\delta$ it holds 
\begin{equation}\label{est1} 
\forall \; x\in\O_\delta,\ t\geq 1,\quad-\delta^\rho+d(x)^\rho+
\min_{(\partial \O_\delta\setminus \partial\O) \times [0,t]}u 
\leq u(x,t)\leq \delta^\rho-d(x)^\rho+\max_{(\partial \O_\delta\setminus
\partial\O)\times [0,t]}u.\end{equation} 
\end{proposition} 

\begin{proof} By Lemma \ref{barrier} there exists $\bar\delta>0$ such that, for 
$x\in \O_{\bar\delta}$, 
$$F[1-d(x)^\rho]\geq M,\quad F[d(x)^\rho-1]\leq-M,\quad 
\text{ with }\;M:=2\|u_0\|_\infty+\|l\|_\infty.$$
Owing to Lemma \ref{liap}, up to reducing $\bar\delta$ if needed, we also 
have that $F[-d(x)^{-1}]\leq 0$
and  $F[d(x)^{-1}]\geq 0$ in $\O_{\bar\delta}$. Fix $t\geq1$ and, for 
$\delta<\bar\delta$ and $\eps>0$, let us define in $\O_{\bar\delta}\times[0,t]$  the 
following function:
\[v^\eps(x,s):=\max_{ \partial \O_\delta\setminus \partial\O\times [0,t]}u+ 
(1-d(x)^\rho)-1+\delta^\rho+\eps d(x)^{-1}-2 (s-t)\frac{\|u_0\|_\infty}{t}.\]
This is a $C^{2,1}$ function which, by \eqref{HF}, satisfies in 
$\O_\delta\times(0,t]$,
\[v^\eps_s+H[v^\eps]\geq  -2\frac{\|u_0\|_\infty}{t}+F[1-d(x)^\rho]-\|l\|_\infty
-\eps F[-d(x)^{-1}]> -2\|u_0\|_\infty +M-\|l\|_\infty=0.  \]
Moreover $v^\eps$ fulfills the boundary condition 
\[\liminf_{x\to \partial\O} v^\eps(x,s)d(x)=\eps\quad \text{ uniformly 
in }s\in[0,t], \]  
and, for $x\in\partial\O_\delta\setminus \partial \Omega$ and $s\in[0,t]$, 
\[v^\eps(x,s)=\max_{ \partial\O_\delta\setminus \partial \Omega\times
[0,t]}u+\eps \delta^{-1}+2 
(t-s)\frac{\|u_0\|_\infty}{t}\geq \max_{ \partial\O_\delta\setminus \partial
\Omega\times [0,t]}u\geq 
u(x,s),\]
as well as the initial condition
\[v^\eps(x,0)\geq \max_{ \partial\O_\delta\setminus \partial \Omega\times [0,t]}u+2\|u_0\|_\infty \geq 
\max_{ \partial\O_\delta\setminus \partial \Omega}u_0+2\|u_0\|_\infty\geq 
u_0(x).\]
Then by the comparison principle given by Lemma \ref{lem:comparison} applied
with 
$\lambda =1$ (recall that $u$ is bounded in $\om\times[0,t]$) we obtain
\[u(x,s)\leq v^\eps(x,s)\qquad \text{for }x\in\O_\delta,\ s\in [0,t], \] 
which, computed at $s=t$ and with $\eps\to 0$, eventually implies the 
second inequality in \eqref{est1}.

To obtain the first inequality in \eqref{est1} one argues analogously: define
the function \[w^\eps(x,s)=\min_{ \partial\O_\delta\setminus \partial \O\times
[0,t]}u- 1+d(x)^\rho+1-\delta^\rho-\eps   d(x)^{-1}+2
(s-t)\frac{\|u_0\|_\infty}{t},\]
and, after observing that \[w^\eps_s+H[w^\eps]\leq  2\frac{\|u_0\|_\infty}{t}+H[
d(x)^\rho-1]+\eps F[-d(x)^{-1}]\leq 2\|u_0\|_\infty-M\leq 0, \] one concludes
as before.
\end{proof}

We now turn to the ergodic problem 
 \begin{equation}
 \label{cell}   
H[\chi]=c,\qquad   \text{in }\O.
\end{equation}
The solution $\chi$ will be used to derive the large time behavior for the
Cauchy problem \eqref{Cauchy}.

It has been proved in  \cite{bcr} that, under the same standing assumptions as
here, there 
exists a unique constant $c$ for which \eqref{cell}  admits a 
solution $\chi\in C^{2}(\O)$ satisfying \eqref{BC}. Moreover $\chi$ is unique up 
to additive constants and actually satisfies the stronger condition
\begin{equation}
\label{log}
\chi(x)=o(-\log(d(x)))
 \quad \text{ as } \; x\to \partial\Omega. \end{equation}
However, such condition is not sufficient for our purpose, but we need
boundedness. This is provided by the following. 
\begin{proposition}\label{chibounded}
Let $c\in\R$ and $\chi\in C^2(\O)$ be the solution to \eqref{cell} satisfying 
\eqref{log}. Then $\chi\in L^\infty(\O)$. In particular, for every 
$\rho\in(0,1-\gamma)$, there exists $\bar\delta\in (0,1)$ such that, for every  
$\delta\leq\bar\delta$ and $x\in\O_\delta$, there holds
\begin{equation}\label{estchi} 
\min_{\partial\O_\delta\setminus \partial\O }\chi -\delta^\rho+d(x)^\rho\leq 
\chi(x )\leq \max_{\partial\O_\delta\setminus\partial \O}\chi+ 
\delta^\rho-d(x)^\rho.\end{equation}
\end{proposition} 
 
\begin{proof}
First of all notice that $\chi(x)-ct$ is an entire solution to
the 
first equation in \eqref{Cauchy}.
Set $M=2|c|+\|l\|_\infty$ and consider the associated quantity $ \delta$ given 
by Lemma \ref{barrier}. Then 
$$-H[1-d(x)^\rho]\leq H[d(x)^\rho-1]\leq-2|c|\quad\text{for }x\in \O_{ 
\delta}.$$ 
Possibly decreasing $ \delta$, we can also assume  that $-F[ 
d(x)^{-1}]\leq F[-d(x)^{-1}]\leq 0$ in $\O_{ \delta}$ by Lemma 
\ref{liap}.

Take now $t_\eps<0$ such that  \[-|c| t_\eps\geq   
\max_{x\in\overline\O_\delta} (\chi(x)-\eps 
d(x)^{-1})-\max_{\partial\O_\delta\setminus\partial \O} \chi. \] Note that the 
first maximum above exists due to the fact that $\chi$ satisfies \eqref{log}.

For  $\eps>0$ define in $\ol{\O}_\delta\times[t_\eps,0]$  the function
\[v^\eps(x,s)=\max_{ \partial\O_\delta\setminus\partial\O}\chi+ (1-d(x)^\rho)-1+\delta^\rho+\eps  
d(x)^{-1}- 2|c|s.\]
Then $v^\eps$ is in $C^{2,1}(\O_\delta\times[t_\eps,0])$ and satisfies the 
boundary condition \[\liminf_{x\to \partial\O} v^\eps(x,t)d(x)=\eps\geq 0\]  
uniformly in $t\in [t_\eps,0]$. Moreover 
\[v^\eps_s+H[v^\eps]\geq  -2|c| +H[1-d(x)^\rho]-\eps F[-d(x)^{-1}]\geq  0.  \]
Finally at $x\in \partial\O_\delta\setminus\partial\O$, for $s\in[t_\eps,0]$, 
\[v^\eps(x,s)=\max_{ \partial\O_\delta\setminus\partial \O}\chi+\eps 
\delta^{-1}-2|c|s\geq \max_{ \partial\O_\delta\setminus\partial \O 
}\chi-|c|s\geq  \chi(x)-cs \] and for all $x\in \overline{\O}_\delta$ 
\[v^\eps(x,t_\eps)\geq \max_{\partial\O_\delta\setminus\partial \O }\chi+\eps d(x)^{-1}-2|c| t_\eps\geq \chi(x)-ct_\eps\] by our choice of $t_\eps$. 
Then by the parabolic comparison principle (see Lemma \ref{lem:comparison}), we get that  for every
$\eps>0$
\[\chi(x)-cs\leq v^\eps(x,s)\qquad x\in\O_\delta,\ s\in [t_\eps, 0]. \] 
Computing the previous inequality at $s=0$, and letting $\eps \to 0$, we get the 
right hand side of the inequality \eqref{estchi}.

The other side is obtained with similar arguments, by considering the function
\[w^\eps(x,s)=\min_{ \partial\O_\delta\setminus\partial 
\O}\chi-(1-d(x)^\rho)+1-\delta^\rho-\eps  d(x)^{-1}+ 2|c|s \] for $s\in 
[t^\eps,0]$, where $t^\eps<0$ satisfies \[|c|t^\eps\leq 
\min_{x\in\overline\O_\delta} (\chi(x)+\eps 
d(x)^{-1})-\min_{\partial\O_\delta\setminus\partial \O}\chi. \]

\end{proof}
\begin{remark}\upshape \label{regolaritachi}
In dimension $N=1$ condition \eqref{estchi} implies that $\chi$ is
H\"older-continuous in $\O$ with H\"older exponent up to $1-\gamma$. 
Indeed, if $\Omega=(a,b)$ then the sets $\om_\delta$ are composed by two
disjoint intervals and this allows one to split the estimates \eqref{estchi}
into the following two sets of estimates:
%
\[ 
\forall\; a<x<y<a+\overline\delta,\quad  
\chi(y) - (y-a)^\rho +(x-a)^\rho\leq \chi(x)\leq \chi(y)+
(y-a)^\rho-(x-a)^\rho,
\]
\[ 
\forall\; b-\overline\delta<y<x<b,\quad  
\chi(y) -(b-y)^\rho+(b-x)^\rho\leq \chi(x)\leq  \chi(y)+
(b-y)^\rho-(b-x)^\rho. \]
Combining this with standard interior regularity for elliptic equations, we get
that $\chi\in C^{\rho}(\O)$ for every $\rho\in (0, 1-\gamma)$.

 In general we cannot expect more than
H\"older regularity for $\chi$ under the hypothesis \eqref{invariance}. In
particular, if the constant $\gamma$ there is strictly positive, $\chi$ may not
be Lipschitz-continuous in $\Omega$, as shown by the following example.
Let $\Omega$ be the interval $(0,1)$, the set $A$ to be a singleton,  $a$ to be
a smooth function
such that $a>0$ for $x\in (0,1)$,  $a(x)=x^{2\beta} $ for $x$ in a neighborhood
of $0$ and $a(x)=(1-x)^{2\beta}$ for $x$ in a neighborhood of $1$ with
$2\beta>1$, $b$ be a smooth function  such that $b(0)=b(1)=0$, and $l$ be a
smooth function such that $l(0)\not=l(1)$.    
Assume moreover that there exist $\delta, k, \gamma$, with  $0<\gamma<2\beta-1$ such that 
\[\forall \; x\in (0, \delta),\quad b(x)\geq k x^\gamma \qquad \text{ and 
}\qquad
\forall \; x\in (1-\delta, 1),\quad  b(x)\leq -k (1-x)^\gamma.\]   Note that
this condition is exactly condition \eqref{invariance}, and that is compatible
with the assumption $b(0)=b(1)=0$, since $\gamma>0$. 
The solution $\chi$ of \eqref{cell} solves \[\chi''(x)=
a(x)^{-1} (-b(x)\chi'(x)-l(x)-c), \qquad x\in (0,1).\] Define $G(x)=
-b(x)\chi'(x)-l(x)-c$. 
Assume by contradiction that $\chi$ is Lipschitz-continuous in $(0,1)$, so in
particular there exists $C>0$ such that $|\chi'|\leq C$. 
By our assumptions on the coefficients,  \[\lim_{x\to 0^+} G(x)= -l(0)-c\neq \lim_{x\to 1^{-}} G(x)=-l(1)-c.\]  Then necessarily,  either $\lim_{x\to 0^+} G(x)\neq 0$ or $\lim_{x\to 1^{-}} G(x)\neq 0$. 
Assume to fix the ideas that $\lim_{x\to 0^+} G(x) \neq 0$ (the other case is
completely analogous), then
in a neighborhood of $0$ we get that 
$\chi''(x) \approx   a(x)^{-1}=x^{-2\beta}$. So $\chi'(x) \approx  
x^{-2\beta+1}+C$, and this is in
contradiction with the Lipschitz-continuity of $\chi$ because $1-2\beta<0$.
\end{remark}

\section{Convergence result} \label{conv}

In this last section we show that solutions of the Cauchy problem share the
same large time behavior. This will imply our main result.

\begin{theorem}\label{LTB}
 Let $u_1,u_2\in C^{2,1}(\Omega\times (0, +\infty))\cap 
C(\Omega\times [0,+\infty))$ be two solutions to the first equation in
\eqref{Cauchy} 
which satisfy the boundary control \eqref{BC}, 
such that 
$u_1-u_2$ is bounded and \begin{equation}\label{assnuova}\forall\tau>0,
\ K\subset\subset\O, \ \ \ \exists C>0,\ \theta\in (0, 1], \quad
\|u_1-u_2\|_{C^{2+\theta,1+\theta/2}(K\times(\tau,+\infty))}\leq C.\end{equation}
Then, as $t\to+\infty$, $u_1-u_2$ converges to a constant uniformly 
in $\O$.
\end{theorem}

\begin{proof}
Consider the difference 
function $w:=u_1-u_2$, which is bounded by hypothesis. 
By Lemmas \ref{liap} and \ref{barrier}, there exist $\delta\in(0,1)$ and
$\rho\in (0, 1-\gamma)$  such that $F[-d^{-1}]\leq-1$ and $F[d^\rho-1]\leq-
2\|w\|_\infty$ in $\om_\delta$. For $t>0$ define
$$m(t):=\min_{x\in\O\setminus\O_\delta}w(x,t).$$
Let $(t_n)_{n\in\N}$ be such that 
$$\limn t_n=+\infty,\qquad
\limn m(t_n)=\liminf_{t\to+\infty}m(t)=:\tilde m\,.
$$
Our aim is to show that $w=u_1-u_2\to\tilde m$ uniformly in $\O$ as 
$t\to+\infty$.

\medskip
Step 1.{ \em $w(\cdot,\cdot+t_n)\to\tilde m$ as 
$n\to\infty$, locally uniformly in $\O\times(-\infty,0]$.}

\smallskip\noindent
By assumption \eqref{assnuova}, the functions $(w(\cdot,\cdot+t_n))_{n\in\N}$ 
and their derivatives $\de_t$, $D$, $D^2$ are locally uniformly bounded in 
$\O\times\R$. 
%
Moreover, by \eqref{HF}, $w$ 
solves $w_t +F[w]\geq 0$. Thus, as 
$n\to\infty$, $w(\cdot,\cdot+t_n)$ converges locally uniformly (up to 
subsequences) to a supersolution $\tilde w$ of the same equation in 
$\O\times\R$, which satisfies in addition
\begin{equation}\label{minw}
\min_{(\O\setminus\O_\delta)\times\R}\tilde w= 
\min_{(\O\setminus\O_\delta)\times\{0\}}\tilde w=\tilde m\,.
\end{equation}
For given $T\in\R$ and $\e>0$, the function $v$ defined by
$$v(x,t):=\e(t-T-d(x)^{-1})+\tilde m$$
satisfies $v_t +F[v]\leq 0$ in $\O_\delta\times\R$. 
Moreover, for $x\in\O_\delta$, $v(x,t)<\tilde m$ if $t\leq T$, and 
$v(x,t)<\inf\tilde w$ if $t$ is less than some $t_\e\in\R$, that we can suppose 
to be smaller than $T$. We can therefore apply the comparison principle between 
$v$ and $\tilde w$ in $\O_\delta\times[t_\e,T]$ and deduce in particular that
$$\forall \; x\in \O_\delta,\quad
\tilde w(x,T)\geq v(x,T)=-\e d(x)^{-1}+\tilde m\,.$$
By the arbitrariness of $\e$ and $T$, we then infer that $\tilde w\geq\tilde m$ 
in $\O_\delta\times\R$. Eventually, by~\eqref{minw}, $\tilde w$ 
attains its global minimum $\tilde m$ somewhere in $\O\setminus\O_\delta$ at 
time $0$.
The parabolic strong maximum principle then implies that $\tilde w=\tilde m$ 
for all $t\leq0$. We have shown that $w(\cdot,\cdot+t_n)\to\tilde m$ as 
$n\to\infty$, locally uniformly in $\O\times(-\infty,0]$. 

\medskip
Step 2.{ \em  $w(\cdot,t_n)\to\tilde m$ uniformly in $\O$ as 
$n\to\infty$.}

\smallskip\noindent
Define $m_{n,\delta}:= \min_{(y,s)\in(\O\setminus\O_\delta)\times 
[-1,0]}w(y,t_n+s)$.
Observe that by Step 1, $\lim_n m_{n,\delta}=\tilde m$. 
For given $n\in\N$ consider the function  
\[v(x,t)=m_{n,\delta}- 1+d(x)^\rho+1-\delta^\rho-\eps  d(x)^{-1}+2
t\|w\|_\infty.\]
Using \eqref{HF} and recalling our definition of $\delta$, we find that 
\[v_t+F[v]\leq 2  \|w\|_\infty+ F[d(x)^\rho-1]+\eps F[ -d(x)^{-1}]\leq 0, \qquad
x\in \O_\delta,\ t\in (-1,0). \] 
Moreover $\limsup_{x \to \partial \O} v(x,t) d(x)\leq 0$, uniformly in $t\in
(-1,0)$, and
\[\forall x\in \O_\delta,\quad v(x,-1)\leq m_{n,\delta}-2 \|w\|_\infty\leq
-\|w\|_\infty\leq w(x,t_n-1).\] 
Finally, if $d(x)=\delta$ and $t\in [-1,0]$, then $v(x,t)\leq 
m_{n,\delta}$. 
Then the standard comparison principle yields
\[\forall x\in\O_\delta,\ s\in [-1,0],\quad
w(x,t+t_n)\geq m_{n,\delta}+d(x)^\rho-\delta^\rho-\eps d(x)^{-1}+2
s\|w\|_\infty.\]Letting 
$\eps\to 0$ and computing the inequality at $s=0$, we obtain 
\[w(x,t_n)\geq m_{n,\delta}+d(x)^\rho-\delta^\rho\geq m_{n,\delta}-\delta^\rho \qquad x\in\O_\delta.\]

This implies that \[m_{n,\delta}-\delta^\rho\leq \inf_{x\in\O}w(x,t_n)\leq m_{n,\delta}.\]
So, letting $n\to +\infty$, by the arbitrariness of $\delta$ we get 
\[\liminf_n\inf_{x\in\O}w(x,t_n)=\tilde m.\]

Again by Step 1, as $n\to +\infty$ we have
\[ M_{n,\delta}:= - \max_{(y,s)\in(\O\setminus\O_\delta)\times 
[-1,0]}w(y,t_n+s)=\min_{(y,s)\in(\O\setminus\O_\delta)\times 
[-1,0]}(-w(y,t_n+s))\to -\tilde m.\]   
Hence repeating the same argument as above for $-w$ (exchanging the role of $u_1, u_2$) we get
\[\liminf_n\inf_{x\in\O}(-w(x,t_n))=-\tilde 
m,\qquad\text{i.e.}\qquad
\limsup_n\sup_{x\in\O}w(x,t_n)=\tilde m.\]
This concludes the proof of the step.



\medskip
Step 3.{ \em  $w(\cdot,t)\to\tilde m$ uniformly in $\O$ as 
$t\to+\infty$.}
 
\smallskip\noindent
Define
$$\ul m(t):=\inf_{x\in\O} w(x,t),\qquad
\ol m(t):=\sup_{x\in\O} w(x,t).$$
For any $s>0$, the function $u_2+\ul m(s)$ is a subsolution to the first 
equation 
in \eqref{Cauchy} and lies below $u_1$ at time $t=s$. Hence, Lemma 
\ref{lem:comparison} implies that $u_2+\ul m(s)\leq u_1$ for all $t\geq s$, 
from which we deduce that
$t\mapsto \ul m(t)$ is nondecreasing. Analogously, $u_2+\ol m(s)\geq u_1$ for 
all $t\geq s$, whence $t\mapsto \ol m(t)$ is nonincreasing.
It follows that
$$\lim_{t\to+\infty}\ul m(t)=\limn\ul m(t_n)=\tilde m=
\limn\ol m(t_n)=\lim_{t\to+\infty}\ol m(t).$$
This concludes the proof.
\end{proof}


\begin{corollary}\label{main}
Let $u_0\in L^\infty(\Omega)\cap C(\Omega)$ and let
$u$ be the unique solution to \eqref{Cauchy} satisfying the boundary control
\eqref{BC}.
Then there exists a constant $K$, depending only on $\|u_0\|_\infty$, such that 
\[u(x,t)+ct-\chi(x)+K\to 0 \qquad \text{as $t\to +\infty$, \
uniformly in $x\in\om$},\] where $(c,\chi)$ is the bounded solution to
\eqref{cell} normalized by $\sup\chi=0$. 
\end{corollary}

 \begin{proof} The result is a straightforward application of Theorem \ref{LTB} 
to $u_1(x,t):=u(x,t)$ and
$u_2(x,t):=\chi(x)-ct$, once we check the assumptions.

We recall that $\chi$ is bounded, by Proposition
\ref{chibounded}. Then, by the comparison principle of
Lemma~\ref{lem:comparison}, we have that
\[-\|u_0\|_\infty\leq u(x,t)-(\chi(x)-ct)\leq  \|u_0\|_\infty-\inf \chi. \] 
This implies that $u_1-u_2$ is bounded. 
The same  inequality also implies  that $u(x,t)+ct$ is bounded in 
$\Omega\times (0, +\infty)$. 
So,  both $u(x,t)+ct$  and $\chi(x)$ are  globally bounded solutions to $\tilde u_t+H[\tilde u]=c$. 
Let $\tilde u$ be  a solution to $\tilde u_t+H[\tilde u]=c$, such that $ \tilde 
u \in L^{\infty}(\Omega\times(0, \infty))$. Arguing as in Theorem 
\ref{escauchy}, by Theorem 1.1 in \cite{tw},  we infer that, 
for all $\tau>0$ and all $\Omega'\subset\subset \Omega$, there exist
$\theta,C>0$ such that  $\|D^2\tilde u\|_{ C^{ \theta, 
\theta/2}(\Omega'\times(\tau,+\infty))}\leq C$.  
Hence, using the equation, we eventually find that $\tilde u\in C^{2+ 
\theta, 1+\theta/2}(\Omega'\times(\tau,+\infty))$.  In particular, we have that  
this estimate holds for both $u(x,t)+ct$
and $\chi$, which implies that the hypothesis \eqref{assnuova} is satisfied. 
 \end{proof}

\end{document}